\documentclass{amsart}
\usepackage{amsthm}
\usepackage{amsmath}
\usepackage{mathtools}
\usepackage{amssymb}
\usepackage{hyperref}
\usepackage{tikz}
\usetikzlibrary{decorations.markings}
\usepackage{enumitem}
\usepackage{amsfonts}
\usepackage{hyperref}
\usepackage{todonotes}
\usepackage[utf8]{inputenc}

\newcommand\precdot{\mathrel{\ooalign{$\prec$\cr
  \hidewidth\raise0.001ex\hbox{$\cdot\mkern0.6mu$}\cr}}}

\newtheorem{theorem}{Theorem}[section]
\newtheorem{def-prop}[theorem]{Definition-Proposition}
\newtheorem{prop}[theorem]{Proposition}
\newtheorem{conj}[theorem]{Conjecture}
\newtheorem{lemma}[theorem]{Lemma}
\newtheorem{cor}[theorem]{Corollary}

\theoremstyle{definition}
\newtheorem{ex}[theorem]{Example}

\newtheorem{defin}[theorem]{Definition}

\theoremstyle{remark}
\newtheorem*{remark}{Remark}

\DeclareMathOperator{\conv}{Conv}
\DeclareMathOperator{\cay}{Cay}
\DeclareMathOperator{\ac}{AC}
\DeclareMathOperator{\id}{id}
\newcommand{\Z}{\mathbb{Z}}
\newcommand{\rev}{\normalfont{rev}}
\newcommand{\R}{\mathbb{R}}
\renewcommand{\S}{\mathfrak{S}}
\newcommand{\LL}{\mathcal{L}}
\newcommand{\B}{\mathcal{B}}

\title{The hull metric on Coxeter groups}
\author{Christian Gaetz}
\thanks{C.G. is supported by a National Science Foundation Postdoctoral Research Fellowship under Grant No. DMS-2103121.}
\address{Department of Mathematics, Harvard University, Cambridge, MA 02138}
\email{\href{mailto:crgaetz@gmail.com}{{\tt crgaetz@gmail.com}}}

\author{Yibo Gao}
\address{Department of Mathematics, Massachusetts Institute of Technology, Cambridge, MA 02139}
\email{\href{mailto:gaoyibo@mit.edu}{{\tt gaoyibo@mit.edu}}}

\date{\today}

\begin{document}
\begin{abstract}
We reinterpret an inequality, due originally to Sidorenko, for linear extensions of posets in terms of \emph{convex subsets} of the symmetric group $\S_n$. We conjecture that the analogous inequalities hold in arbitrary (not-necessarily-finite) Coxeter groups $W$, and prove this for the hyperoctahedral groups $B_n$ and all right-angled Coxeter groups. Our proof for $B_n$ (and new proof for $\S_n$) use a combinatorial \emph{insertion map} closely related to the well-studied \emph{promotion} operator on linear extensions; this map may be of independent interest. 

We also note that the inequalities in question can be interpreted as a triangle inequalities, so that convex hulls can be used to define a new invariant metric on $W$ whenever our conjecture holds. Geometric properties of this metric are an interesting direction for future research.  
\end{abstract}

\maketitle

\section{Introduction}
Let $G$ be a connected undirected graph with vertices $V(G)$ and let
\[
d:V(G)\times V(G)\rightarrow\Z_{\geq0}
\]
denote the shortest-path distance between vertices, clearly a metric on $V(G)$. A subset $C \subseteq V(G)$ is \emph{convex} if no shortest path between two vertices from $C$ leaves $C$: whenever $u,v \in C$ and $d(u,w)+d(w,v)=d(u,v)$ we must have $w \in C$. For $X \subseteq V(G)$, the \textit{convex hull} $\conv(X)$ is the intersection  (itself clearly convex) of all convex sets containing $X$.

The following definition, which to our knowledge has not been considered before, is central to this work.

\begin{defin}
\label{def:metric-property}
We say a graph $G$ has the \textit{hull property} if for every $u,v,w\in V(G)$,
\begin{equation}
\label{eq:hull-property}
\left|\conv(u,v)\right|\cdot \left|\conv(v,w)\right|\geq \left|\conv(u,w)\right|.
\end{equation}
Equivalently, $G$ has the hull property if and only if
\[
\log \left|\conv(-,-)\right|:V(G)\times V(G)\rightarrow\R_{\geq0}
\]
is a metric on $V(G)$.
It has the \textit{strong hull property} if for every $u,v,w\in V(G)$,
\begin{equation}
\label{eq:strong-hull}
\left|\conv(u,v)\right|\cdot\left|\conv(v,w)\right|\geq\left|\conv(u,v,w)\right|.
\end{equation}
\end{defin}

\subsection{Main results}
For a Coxeter group $W$ we write $\cay(W)$ for the undirected right Cayley graph for $W$ with the generating set of simple reflections (see Section~\ref{sec:coxeter-background}).

\begin{conj}
\label{conj:main-conjecture}
Let $W$ be any Coxeter group, then $\cay(W)$ has the strong hull property.
\end{conj}

In addition to the infinite families of groups for which we prove Conjecture~\ref{conj:main-conjecture} below, we have also verified the conjecture by computer for the finite Coxeter groups of types $G_2, H_3, D_4,$ and $F_4$.

It would be an interesting direction for future research to study whether Conjecture~\ref{conj:main-conjecture} might be extended to arbitrary Tits buildings.

To see the strength of Conjecture~\ref{conj:main-conjecture}, consider the following very special case:

\begin{ex}
Let $W=\S_n$ be the symmetric group, with the adjacent transpositions as simple generators. Let $w \in \S_n$ be any permutation and $w^{\rev}$ the reversed permutation (identifying permutations with their one-line notation $w(1) \ldots w(n)$). Then it follows from observations of Bj\"{o}rner--Wachs \cite{Bjorner-Wachs-gen-quo} that $\left| \conv(\id,w) \right| = \left| [\id,w]_R \right|$ is the number of linear extensions of the corresponding \emph{2-dimensional poset} $P_w$, while $\left| \conv(\id,w^{\rev}) \right|$ is the number of linear extensions of the \emph{complement} $\overline{P_w}=P_{w^{\rev}}$, and furthermore that any 2-dimensional poset is isomorphic to some $P_w$.  Since $\conv(w,w^{\rev})=\S_n$, the hull property implies that for any 2-dimensional poset $P$:
\begin{equation}
\label{eq:sidorenko}
e(P)e(\overline{P}) \geq n!
\end{equation}
where $e(P)$ denotes the number of linear extensions, a result due originally to Sidorenko \cite{Sidorenko}.
\end{ex}

Sidorenko's original proof \cite{Sidorenko} of (\ref{eq:sidorenko}) relied on intricate max-flow min-cut arguments, but the inequality (\ref{eq:sidorenko}) has since been shown to be related to several different topics and techniques in convex geometry and combinatorics. Bollob\'{a}s--Brightwell--Sidorenko \cite{BBS} later gave a convex geometric proof using a known special case of the still-open Mahler Conjecture. In \cite{Second-separable-paper, First-separable-paper} the authors gave a strengthening and new proof related to the algebraic notion of \emph{generalized quotients} \cite{Bjorner-Wachs-gen-quo} in Coxeter groups.  

Theorem~\ref{thm:types-A-and-B} below extends Sidorenko's inequality to all pairs of elements (not just $w,w^{\rev}$) and to the hyperoctahedral group (equivalently, to certain signed posets).  Our proof relies on a new combinatorial \emph{insertion map} for linear extensions (see Definition~\ref{def:insertion-map}), closely related to \emph{promotion} \cite{Schutzenberger}. This combinatorial proof of Sidorenko's inequality gives a new solution to an open problem of Morales--Pak--Panova \cite{Morales-Pak-Panova} different from that in \cite{Second-separable-paper}. In the case of the symmetric group, a combinatorial proof similar to our proof in Section~\ref{sec:permutations} has been obtained independently by Chan--Pak--Panova \cite{Chan-Pak-Panova-alternative}.

\begin{theorem}
\label{thm:types-A-and-B}
Conjecture~\ref{conj:main-conjecture} holds when $W$ is the symmetric group $\S_n$ or hyperoctahedral group $B_n$.
\end{theorem}

We also prove Conjecture~\ref{conj:main-conjecture} for the large class of \emph{right-angled} Coxeter groups using very different methods. Heuristically, these groups are on the opposite extreme among Coxeter groups from the finite groups covered in Theorem~\ref{thm:types-A-and-B}, as the non-commuting products $s_is_j$ have infinite order, as opposed to the small, finite orders seen in finite Coxeter groups. In addition, right-angled Coxeter groups are a natural case in which to verify our conjectured metric because of their ubiquity in geometric group theory (see, for example, the survey \cite{pallavi}).

\begin{theorem}
\label{thm:right-angled-case}
Conjecture~\ref{conj:main-conjecture} holds when $W$ is any right-angled Coxeter group.
\end{theorem}

For both Theorem~\ref{thm:types-A-and-B} and Theorem~\ref{thm:right-angled-case} we give explicit injections realizing the inequality (\ref{eq:strong-hull}), see Theorems~\ref{thm:typeA}, \ref{thm:typeB}, and \ref{thm:right-angled-has-metric}.

It is natural to ask when the inequality (\ref{eq:hull-property}) is in fact an equality. In the case of a triple $(\id,w,ww_0)$ in a finite Weyl group $W$ with longest element $w_0$, equality is achieved when $w$ is \emph{separable} (see \cite{Second-separable-paper, First-separable-paper}) and conjecturally only in this case. For the symmetric group, this corresponds to the case where $P_w$ is a series-parallel poset.

\subsection{Outline}

Section~\ref{sec:coxeter-background} covers needed background on Coxeter groups, convex subsets thereof, and the weak order.  Section~\ref{sec:permutations} proves Conjecture~\ref{conj:main-conjecture} in the case $W=\S_n$ using a new combinatorial \emph{insertion map} for linear extensions; this argument is extended in Section~\ref{sec:signed} to prove the conjecture for the hyperoctahedral group $B_n$.

Section~\ref{sec:right-angled} proves some structural properties of inversion sets in right-angled Coxeter groups and uses these to give an explicit injective proof of Conjecture~\ref{conj:main-conjecture} when $W$ is any right-angled Coxeter group.

Finally, in Section~\ref{sec:graphical}, we demonstrate that the (strong) hull property is a rather special property of graphs by proving that the graph of regions of a graphical hyperplane arrangement has the hull property if and only if it is in fact the reflection arrangement for $\S_n$, in which case the property follows from Theorem~\ref{thm:types-A-and-B}.

\section{Background on Coxeter groups}
\label{sec:coxeter-background}
The material in this section is well known and may be found, for example, in Bj\"{o}rner--Brenti \cite{Bjorner-Brenti}.

\subsection{Coxeter groups and weak order}
A \emph{Coxeter group} is a group $W$ together with a distinguished generating set $S=\{s_1,\ldots,s_r\}$ subject to relations $s_i^2=\id$ for $i=1,\ldots,r$ and $(s_is_j)^{m_{ij}}=\id$ for $i\neq j \in \{1, \ldots, r\}$ for some numbers $m_{ij} \in \{2,3,\ldots\} \sqcup \{\infty\}$.  The $r \times r$ matrix with entries $m_{ij}$ is the \emph{Coxeter matrix}. The elements of $S$ are called \emph{simple reflections}, while their $W$-conjugates are called \emph{reflections}; the set of reflections is denoted $T$.  Given an element $w \in W$, an expression
\[
w=s_{i_1}\cdots s_{i_{\ell}}
\]
of minimal length is called a \emph{reduced word} for $w$, and in this case $\ell=\ell(w)$ is the \emph{length} of $w$.

The \emph{(right) weak order} $(W, \leq_R)$ has cover relations $w \lessdot_R ws$ whenever $\ell(ws)=\ell(w)+1$ and $s \in S$.  A reflection $t \in T$ is a \emph{(left) inversion} of $w$ if $\ell(tw)<\ell(w)$, and the set of inversions of $w$ is denoted $T_L(w)$.  The following facts are well known:
\begin{prop}
\label{prop:weak-order-from-inversions}
For $u,v \in W$ we have $u \leq_R v$ if and only if $T_L(u) \subseteq T_L(v)$.
\end{prop}

\begin{prop}
\label{prop:inversions-from-prefixes}
Let $w=s_{i_1}\cdots s_{i_{\ell}}$ be any reduced word for $w \in W$, then 
\[
T_L(w)=\{s_{i_1}s_{i_2}\cdots s_{i_{\ell'-1}} s_{i_{\ell'}} s_{i_{\ell'-1}} \cdots s_{i_2}s_{i_1} \: | \: 1 \leq \ell' \leq \ell\}.
\]
\end{prop}

\begin{theorem}[Bj\"{o}rner \cite{Bjorner-survey}]
The weak order on any Coxeter group $W$ is a meet-semilattice.
\end{theorem}

We denote the meet operation in $(W,\leq_R)$ by $\land_R$.

\subsection{Convexity in Coxeter groups}
We say a subset $C \subseteq W$ is \emph{convex} if it is convex in the Cayley graph $\cay(W,S)$ (the Hasse diagram of weak order).  For $D \subseteq A \subseteq T$, we write 
\[
W(D,A) \coloneqq \{w \in W \: | \: D \subseteq T_L(w) \subseteq A\}.
\]

\begin{theorem}[Tits \cite{Tits}]
\label{thm:tits-convexity}
A set $C \subseteq W$ is convex if and only if $C=W(D,A)$ for some $D \subseteq A \subseteq T$.
\end{theorem}

\begin{prop}[Bj\"{o}rner and Wachs \cite{Bjorner-Wachs-gen-quo}]
\label{prop:convex-hull-from-inversions}
Let $x_1, \ldots, x_d \in W$, then 
\[
\conv(x_1, \ldots, x_d) =\left\{ w \in W \; \middle| \; \bigcap_{i=1}^d T_L(x_i) \subseteq T_L(w) \subseteq \bigcup_{i=1}^d T_L(x_i)\right\}.
\]
In particular, $\conv(\id,x)=[e,x]_R$ for all $x \in W$.
\end{prop}

\subsection{Parabolic subgroups and quotients}
For $J \subseteq S$, the \emph{parabolic subgroup} $W_J$ is the subgroup of $W$ generated by $J$ (viewed as a Coxeter group with simple generators $J$).

\begin{prop}
\label{prop:parabolic-quotient-facts}
Let $J \subseteq S$, then:
\begin{itemize}
    \item Each left coset $wW_J$ of $W_J$ in $W$ has a unique element of minimal length, denoted $w^J$.  The set $\{w^J \: | \: w \in W\}$ is called the \emph{parabolic quotient} and is denoted $W^J$.
    \item Each $w \in W$ can be uniquely written $w=w^J w_J$ with $w^J \in W^J$ and $w_J \in W_J$.  This decomposition satisfies $\ell(w)=\ell(w^J)+\ell(w_J)$.
    \item Therefore, we have $w^J \leq_R u$ for all $u \in wW_J$.
\end{itemize}
\end{prop}

\section{The symmetric group}\label{sec:permutations}

\subsection{The insertion map}
Recall that the symmetric group $\S_n$ is a Coxeter group generated by $S=\{s_1,\ldots,s_{n-1}\}$ where $s_i$ is the simple transposition $(i\ i+1)$. The set of reflections is $T=\{t_{i,j}\:|\: 1\leq i<j\leq n\}$ where $t_{i,j}=(i\ j)$. For $w\in\S_n$, $t_{i,j}\in T_L(w)$ if $w^{-1}(i)>w^{-1}(j)$, for $1\leq i<j\leq n$.

In this section, let $P$ be a poset on $n$ elements $\{p_1,\ldots,p_n\}$. A \textit{linear extension} is an order-preserving bijection $\lambda:P\rightarrow[n]$ from the poset $P$ to the $n$-element chain $[n]$. Denote the set of linear extensions of $P$ by $\LL(P)$. Each $\lambda\in \LL(P)$ has a naturally associated permutation $\pi_{\lambda}$ where $\pi_{\lambda}(i)=j$ if $\lambda(p_j)=i$. 

\begin{defin}
\label{def:insertion-map}
We now define a map $f_P:\S_n\rightarrow \LL(P)$, called the \textit{insertion map}, from the symmetric group $\S_n$ to the set of linear extensions of any poset $P$ on $n$ elements. This map will be defined recursively, and will resemble promotion (see Section 3.20 of \cite{ec1}). 

Let $w\in\S_n$, its image $\lambda=f_P(w)$ is determined as follows: At step $k$, for $k=1,2,\ldots,n$, we start with $\lambda(p_{w(k)})=k$. If at any time we have $\lambda(p_a)=k$ where $p_a$ is not a maximal element of the induced subposet $P$ on $p_{w(1)},\ldots,p_{w(k)}$, find the smallest value $\lambda(p_b)$ where $p_b>_Pp_a$ and  $\lambda(p_b)$ has already been defined, and swap the values of $p_a$ and $p_b$. Continue this procedure to move the value $k$ up the poset, until $\lambda^{-1}(k)$ is a maximal element of the induced subposet $P$ on $p_{w(1)},\ldots,p_{w(k)}$. 
\end{defin}

\begin{ex}
We show an example of the insertion map defined above. Let $P$ be as in Figure~\ref{fig:insertion-typeA-example} and $w=45312$. Figure~\ref{fig:insertion-typeA-example} shows how the insertion map works at each step, where the final result $\lambda=f_P(w)$ is provided at the end whose corresponding permutation is $\pi_{\lambda}=52134$. 
\begin{figure}[h!]
\centering
\begin{tikzpicture}[scale=0.900000000000000]
\node at (0,0) {$\bullet$};
\node at (0,2) {$\bullet$};
\node at (0.866000000000000,0.5) {$\bullet$};
\node at (0.866000000000000,1.5) {$\bullet$};
\node at (-0.200000000000000,1) {$\bullet$};
\node[left] at (0,2) {$p_4$};
\node[left] at (0,0) {$p_5$};
\node[right] at (0.866000000000000,0.5) {$p_2$};
\node[right] at (0.866000000000000,1.5) {$p_3$};
\node[left] at (-0.200000000000000,1) {$p_1$};
\draw(0,0)--(0.866000000000000,0.5)--(0.866000000000000,1.5)--(0,2)--(-0.200000000000000,1)--(0,0);
\node at (2.80000000000000,0) {$\bullet$};
\node at (2.80000000000000,2) {$\bullet$};
\node at (3.66600000000000,0.5) {$\bullet$};
\node at (3.66600000000000,1.5) {$\bullet$};
\node at (2.60000000000000,1) {$\bullet$};
\node[left] at (2.80000000000000,2) {$p_4$};
\node[left] at (2.80000000000000,0) {$p_5$};
\node[right] at (3.66600000000000,0.5) {$p_2$};
\node[right] at (3.66600000000000,1.5) {$p_3$};
\node[left] at (2.60000000000000,1) {$p_1$};
\draw(2.80000000000000,0)--(3.66600000000000,0.5)--(3.66600000000000,1.5)--(2.80000000000000,2)--(2.60000000000000,1)--(2.80000000000000,0);
\node at (5.60000000000000,0) {$\bullet$};
\node at (5.60000000000000,2) {$\bullet$};
\node at (6.46600000000000,0.5) {$\bullet$};
\node at (6.46600000000000,1.5) {$\bullet$};
\node at (5.40000000000000,1) {$\bullet$};
\node[left] at (5.60000000000000,2) {$p_4$};
\node[left] at (5.60000000000000,0) {$p_5$};
\node[right] at (6.46600000000000,0.5) {$p_2$};
\node[right] at (6.46600000000000,1.5) {$p_3$};
\node[left] at (5.40000000000000,1) {$p_1$};
\draw(5.60000000000000,0)--(6.46600000000000,0.5)--(6.46600000000000,1.5)--(5.60000000000000,2)--(5.40000000000000,1)--(5.60000000000000,0);
\node at (8.40000000000000,0) {$\bullet$};
\node at (8.40000000000000,2) {$\bullet$};
\node at (9.26600000000000,0.5) {$\bullet$};
\node at (9.26600000000000,1.5) {$\bullet$};
\node at (8.20000000000000,1) {$\bullet$};
\node[left] at (8.40000000000000,2) {$p_4$};
\node[left] at (8.40000000000000,0) {$p_5$};
\node[right] at (9.26600000000000,0.5) {$p_2$};
\node[right] at (9.26600000000000,1.5) {$p_3$};
\node[left] at (8.20000000000000,1) {$p_1$};
\draw(8.40000000000000,0)--(9.26600000000000,0.5)--(9.26600000000000,1.5)--(8.40000000000000,2)--(8.20000000000000,1)--(8.40000000000000,0);
\node at (11.2000000000000,0) {$\bullet$};
\node at (11.2000000000000,2) {$\bullet$};
\node at (12.0660000000000,0.5) {$\bullet$};
\node at (12.0660000000000,1.5) {$\bullet$};
\node at (11.0000000000000,1) {$\bullet$};
\node[left] at (11.2000000000000,2) {$p_4$};
\node[left] at (11.2000000000000,0) {$p_5$};
\node[right] at (12.0660000000000,0.5) {$p_2$};
\node[right] at (12.0660000000000,1.5) {$p_3$};
\node[left] at (11.0000000000000,1) {$p_1$};
\draw(11.2000000000000,0)--(12.0660000000000,0.5)--(12.0660000000000,1.5)--(11.2000000000000,2)--(11.0000000000000,1)--(11.2000000000000,0);
\draw[->] (1.50000000000000,1)--(2.00000000000000,1);
\draw[->] (4.30000000000000,1)--(4.80000000000000,1);
\draw[->] (7.10000000000000,1)--(7.60000000000000,1);
\draw[->] (9.90000000000000,1)--(10.4000000000000,1);
\node[below] at (0.1,2) {\textbf{1}};
\node[below] at (2.90000000000000,2) {\textbf{2}};
\node[above] at (2.90000000000000,0) {\textbf{1}};
\node[below] at (5.70000000000000,2) {\textbf{3}};
\node[above] at (5.70000000000000,0) {\textbf{1}};
\node[left] at (6.46600000000000,1.5) {\textbf{2}};
\node[below] at (8.50000000000000,2) {\textbf{4}};
\node[above] at (8.50000000000000,0) {\textbf{1}};
\node[left] at (9.26600000000000,1.5) {\textbf{2}};
\node[right] at (8.20000000000000,1) {\textbf{3}};
\node[below] at (11.3000000000000,2) {\textbf{5}};
\node[above] at (11.3000000000000,0) {\textbf{1}};
\node[left] at (12.0660000000000,1.5) {\textbf{4}};
\node[right] at (11.0000000000000,1) {\textbf{3}};
\node[left] at (12.0660000000000,0.5) {\textbf{2}};
\end{tikzpicture}
\caption{An example of the insertion map.}
\label{fig:insertion-typeA-example}
\end{figure}
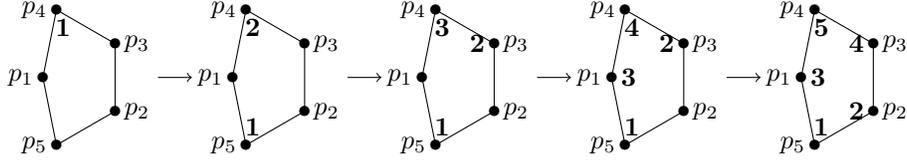
\end{ex}

After step $n-1$ of our construction of $f_P(w)$, only one element $p_i$ has not been assigned a value and we have a linear extension $\lambda'$ of $P\setminus\{p_i\}$. Let us write $\varphi_{p_i}:\LL(P\setminus\{p_i\})\rightarrow\LL(P)$ for the map used in the last step: assign value $n$ to $p_i$, if there exists $p_j$ covering $p_i$, find such $p_j$ with the smallest $\lambda'(p_j)$ and swap the values at $p_i$ and $p_j$, and continue this process until the value $n$ reaches a maximal element in $P$, giving us a linear extension. 

Recall that a \textit{promotion chain} (see \cite{ec1}) of a linear extension $\lambda\in \LL(P)$ is a saturated chain $t_1>t_2>\cdots>t_{\ell}$ from a maximal element to a minimal element in $P$ such that
\begin{itemize}
\item $t_1=\lambda^{-1}(n)$ is the element with the largest $\lambda$-value, and
\item $t_{i+1}$ is the element covered by $t_i$ with the largest $\lambda$-value.
\end{itemize}

\begin{lemma}\label{lem:promotion}
The map $\varphi_{p_i}:\LL(P\setminus\{p_i\})\rightarrow\LL(P)$ is injective, with image consisting of those $\lambda\in\LL(P)$ whose promotion chain passes through $p_i$.
\end{lemma}
\begin{proof}
Take $\lambda\in\LL(P)$ and let $t_1>t_2>\cdots>t_{\ell}$ be its promotion chain. In order for $\lambda$ to come from some $\lambda'\in\LL(P\setminus\{p_i\})$, we invert the process of swapping values: as $\lambda(t_1)=n$, erase the value at $t_1$, find the element covered by $t_1$ with the largest $\lambda$-value, which is $t_2$, erase the value at $t_2$ and give $\lambda(t_2)$ to $t_1$, continue this process until $p_i$ has its value erased. This inverse process can succeed if and only if $p_i$ is part of the promotion chain of $\lambda$. 
\end{proof}

\subsection{Inversion posets}
\begin{defin}
The \textit{inversion poset} $P_u$ of a permutation $u\in\S_n$ is a poset on $n$ elements $p_1,\ldots,p_n$ such that $p_i<p_j$ if $i<j$ and $u^{-1}(i)<u^{-1}(j)$. 
\end{defin}

The following observation of Bj\"{o}rner--Wachs \cite{BW-lin-ext} is the fundamental link between the notions of convex subsets of the symmetric group and linear extensions of posets, and it is in analogy with this proposition that definitions related to posets can be extended to other Coxeter groups. This perspective on extending definitions and results about posets to other Coxeter groups was studied by Reiner \cite{Reiner-signed-posets} and more recently by the authors \cite{gaetz2020balance}.

\begin{prop}[Bj\"{o}rner and Wachs \cite{BW-lin-ext}]\label{prop:BW-lin-ext}
The interval $[\id,w]_R$ in the weak order is exactly $\{\pi_\lambda\:|\:\lambda\in \LL(P_w)\}.$
\end{prop}

We will also be using the permuted inversion poset $wP_u$, where $w\in\S_n$ acts on $P_u$ by permuting the labels. In $wP_u$, $p_{w(i)}<p_{w(j)}$ if and only if $p_i<p_j$ in $P_u$. Note that $P_u$ is always \textit{naturally labeled}, meaning that $p_i\mapsto i$ is a linear extension, but $wP_u$ is not necessarily naturally labeled. 

The inversion poset can be generalized as follows.
\begin{defin}
Let $U\subset\S_n$, then the poset $P_U$ is defined on $n$ elements $p_1,\ldots,p_n$ such that $p_i<p_j$ if $u^{-1}(i)<u^{-1}(j)$ for all $u\in U$.
\end{defin}

Therefore, the inversion poset $P_u$ is just $P_{\{\id,u\}}$ and the permuted inversion poset $wP_{u}$ is $P_{\{w, wu\}}$. The following proposition resembles Proposition~\ref{prop:BW-lin-ext}, we include the proof for convenience. 

\begin{prop}\label{prop:generalized-BW-lin-ext}
For $U\subset\S_n$, the convex set $\conv(U)$ is exactly $\{\pi_{\lambda}\:|\:\lambda\in\LL(P_U)\}$.
\end{prop}
\begin{proof}
By Proposition~\ref{prop:convex-hull-from-inversions}, $\conv(U)$ consists of those $w\in \S_n$ such that
\[
\bigcap_{u\in U} T_L(u)\subset T_L(w)\subset\bigcup_{u\in U}T_L(u).
\]
The set $\bigcap_{u\in U} T_L(u)$ consists of those $t_{i,j}$ where $i<j$ and $u^{-1}(i)>u^{-1}(j)$ for all $u\in U$; likewise, the set $\bigcup_{u\in U}T_L(u)$ does not contain any $t_{i,j}$ where $i<j$ and $u^{-1}(i)<u^{-1}(j)$ for all $u\in U$. So $w\in\conv(U)$ if and only if $w^{-1}(i)<w^{-1}(j)$ for all $i,j$ such that $u^{-1}(i)<u^{-1}(j)$ for all $u\in U$. This corresponds exactly to how $P_U$ is constructed. 
\end{proof}

\begin{theorem}\label{thm:typeA}
For any $w,u\in\S_n$, we have an injection
$$\LL(P_{\{\id,w,wu\}})\hookrightarrow \LL(P_w)\times\LL(wP_u),\text{ given by } \lambda\mapsto \big(f_{P_w}(\pi_\lambda),f_{wP_u}(\pi_\lambda)\big).$$
\end{theorem}
\begin{proof}
Use induction on $n$, the base case $n=1$ being trivial. Let $P=P_{\{\id,w,wu\}}$ and let $A$ be the set of maximal elements of $P$, which is an antichain in $P$. We have the following partition $$\LL(P)=\bigsqcup_{p_a\in A}\{\lambda\in\LL(P)\:|\:\lambda(p_a)=n\}\simeq\bigsqcup_{p_a\in A}\LL(P\setminus\{p_a\}).$$

Fix $p_a\in A$ and consider the restriction $$w':[n]\setminus\{w^{-1}(a)\}\rightarrow [n]\setminus\{a\}$$ of $w$, viewed as a permutation on $n-1$ elements, by keeping the relative order of indices. In this way, the inversion poset $P_{w'}$ is $P_{w}\setminus\{p_a\}$. Also consider the restriction $$u':[n]\setminus\{u^{-1}w^{-1}(a)\}\rightarrow[n]\setminus\{w^{-1}(a)\}$$ so that the inversion poset $w'P_{u'}$ is $wP_u\setminus\{p_a\}$ and $P_{\{\id,w',w'u'\}}$ is $P\setminus\{p_a\}$. 

For $\lambda$ with $\lambda(p_a)=n$, the insertion map $f_{P_w}$ can be identified with the composition $\varphi_{p_a} \circ f_{P_{w'}}$. As a result, the map $\lambda\mapsto\big(f_{P_w}(\pi_{\lambda}),f_{wP_u}(\pi_{\lambda})\big)$ of interest of this theorem can be written as
$$\LL(P)\simeq\bigsqcup_{p_a\in A}\LL(P\setminus\{p_a\})\hookrightarrow\bigsqcup_{p_a\in A}\LL(P_w\setminus\{p_a\})\times\LL(wP_u\setminus\{p_a\}) \rightarrow\LL(P_w)\times\LL(wP_u)$$
where the middle injection is given by the induction hypothesis on $w',u'$. It suffices to show that the last map provided by $\varphi$ is an injection.

Assume to the contrary that we have $\alpha\in\LL(P_w\setminus\{p_a\})$, $\beta\in\LL(wP_u\setminus\{p_a\})$, $\alpha'\in\LL(P_w\setminus\{p_{a'}\})$, and $\beta'\in\LL(wP_u\setminus\{p_{a'}\})$ such that $\varphi_{p_a}(\alpha)=\varphi_{p_{a'}}(\alpha')=\lambda$ and $\varphi_{p_a}(\beta)=\varphi_{p_{a'}}(\beta')=\tau$. If $a=a'$, by Lemma~\ref{lem:promotion} we must have $\alpha=\alpha'$ and $\beta=\beta'$. Thus $a\neq a'$ and we can assume without loss of generality that $a<a'$.

Recall from Lemma~\ref{lem:promotion} that since $\varphi_{p_a}(\alpha)=\lambda$, $p_a$ must be in the promotion chain of $\lambda$. Similarly, the promotion chains of $\lambda$ and $\tau$ both contain $a$ and $a'$. In particular, this means that $p_a$ and $p_{a'}$ are comparable in $P_w$ and that $p_{w^{-1}(a)}$ and $p_{w^{-1}(a')}$ are comparable in $P_u$. As $a<a'$, we must have $w^{-1}(a)<w^{-1}(a')$. And, consequently, we have that $(wu)^{-1}(a)<(wu)^{-1}(a')$. By the definition of $P_{\{\id,w,wu\}}$, as $a<a'$, we know $w^{-1}(a)<w^{-1}(a')$, $(wu)^{-1}(a)<(wu)^{-1}(a')$, and $p_a<_P p_{a'}$, contradicting our assumption that $p_a$ was maximal in $P$.
\end{proof}
\begin{cor}\label{cor:typeA}
The Cayley graph $\cay(\S_n)$ has the strong hull property.
\end{cor}
\begin{proof}
By symmetry of the Cayley graph, $\left|\conv(u,v)\right|=\left|\conv(\id,u^{-1}v)\right|$. For any $u,v,w\in\S_n$, to show $\left|\conv(u,v)\right|\cdot\left|\conv(v,w)\right|\geq\left|\conv(u,v,w)\right|$, it suffices to show that $\left|\conv(\id,w)\right|\cdot\left|\conv(w,wu)\right|\geq\left|\conv(\id,w,wu)\right|$, which follows directly from Proposition~\ref{prop:generalized-BW-lin-ext} and Theorem~\ref{thm:typeA}.
\end{proof}

\begin{remark}
An independent construction of an injection demonstrating Sidorenko's inequality (\ref{eq:sidorenko}) was given by Chan--Pak--Panova \cite{Chan-Pak-Panova-alternative}.
\end{remark}

\section{The hyperoctahedral group}\label{sec:signed}
In this section, we work with the family of finite Coxeter groups of type $B$, called the hyperoctahedral groups. We will extend the constructions and results from Section~\ref{sec:permutations} to this setting in order to complete the proof of Theorem~\ref{thm:types-A-and-B}. 

For $i\in\Z_{>0}$, it will sometimes be convenient to write $\bar i$ for $-i$. For $n\in\Z_{>0}$, let $\overline{[n]}$ denote the set $\{-n,\ldots,-1,1,\ldots,n\}$. Let $B_n$ denote the hyperoctahedral group of all signed permutations $w$ of $\overline{[n]}$ such that $w(i)=-w(\bar i)$ for all $i$. We write such a signed permutation $w$ in one-line notation as $w(1)w(2)\cdots w(n)$. For example, $w=4\bar213$ means $w(1)=4$, $w(-1)=-4$, $w(2)=-2$, $w(-2)=2$ and so on. 

The Coxeter group $B_n$ has simple generators $\{t_{i,i+1}\:|\: 1\leq i\leq n-1\}\cup\{t_{1,\bar 1}\}$ where $t_{i,j}$ is the involution $(i\ j)(\bar i\ \bar j)$ if $i\neq\pm j$ and $(i\ j)$ if $i=\bar j$. The set of reflections is $$T=\{t_{i,j},t_{i,\bar j}\:|\: 1\leq i<j\leq n\}\cup\{t_{i,\bar i}\:|\:1\leq i\leq n\}.$$ For $w\in B_n$, $t_{i,j}\in T_L(w)$ if $w^{-1}(i)>w^{-1}(j)$, $t_{i,\bar j}\in T_L(w)$ if $w^{-1}(i)<w^{-1}(\bar j)$, and $t_{i,\bar i}\in T_L(w)$ if $w^{-1}(i)<w^{-1}(\bar i)$. More details can be found in Chapter 8 of Bj\"{o}rner--Brenti \cite{Bjorner-Brenti}.

\begin{defin}\label{def:typeB-poset}
A poset $P$ on $2n$ elements $\{p_1,\ldots,p_n,p_{\bar1},\ldots,p_{\bar n}\}$ is called a \textit{type B poset} if whenever $p_i<_P p_j$, we have $p_{-j}<_P p_{-i}$. 
\end{defin}
\begin{defin}\label{def:typeB-linear-extension}
A \textit{type B linear extension} of a type $B$ poset $P$ is an order-preserving bijection $\lambda:P\rightarrow\overline{[n]}$ such that $\lambda(p_i)=-\lambda(p_{\bar i})$ for all $i\in [n].$
\end{defin}
The set of type $B$ linear extensions is denoted $\B(P)$. Any $\lambda\in\B(P)$ naturally corresponds to an element $\pi_{\lambda}\in B_n$ where $\pi_{\lambda}(i)=j$ if $\lambda(p_j)=i$ for all $i\in\overline{[n]}$ (see Proposition~\ref{prop:typeB-BW-lin-ext} for an analog of Proposition~\ref{prop:generalized-BW-lin-ext}.)

We now define an insertion map for type $B$ posets, which requires more careful treatment than the insertion map for posets as in Section~\ref{sec:permutations}.

\begin{defin}
Let $P$ be a type $B$ poset with $2n$ elements $\{p_1,\ldots,p_n,p_{\bar 1},\ldots,p_{\bar n}\}$. The map $\varphi_{p_k}^P:\B(P\setminus\{p_k,p_{\bar k}\})\rightarrow \B(P)$ is defined as follows:
\begin{itemize}
\item For $\lambda'\in\B(P\setminus\{p_k,p_{\bar k}\})$, we construct $\lambda:P\rightarrow\overline{[n]}$ by first setting $\lambda(p_k)=n$, $\lambda(p_{\bar k})=-n$, and $\lambda(p_j)=\lambda'(p_j)$ if $j\neq \pm k$.
\item If $\lambda(p_a)=n$ but $p_a$ is not a maximal element of $P$, find the $p_b$ greater than $p_a$ with the smallest value for $\lambda(p_b)$ and swap the $\lambda$-values of $p_a$ and $p_b$ and simultaneously swap those of $p_{\bar a}$ and $p_{\bar b}$; iterate this process until $\lambda^{-1}(n)$ is a maximal element of $P$.
\end{itemize}
\end{defin}
\begin{remark}
When we have $\lambda(p_a)=n$ and $p_a$ is not a maximal element of $P$, it is possible that the $p_b$ covering $a$ with the smallest $\lambda$-value happens to be $b=-a$, in which case swapping the values of $p_a$ and $p_b$ and of $p_{\bar a}$ and $p_{\bar b}$ are done with just one swap.
\end{remark}
We will write $\varphi_{p_k}$ for $\varphi_{p_k}^P$ for simplicity when there is no confusion. Note that the two maps $\varphi_{p_k}^P$ and $\varphi_{p_{-k}}^P$ from $\B(P\setminus\{p_k,p_{\bar k}\})$ to $\B(P)$ do not coincide. 

\begin{defin}\label{def:promotion-chain}
For $\lambda\in\B(P)$, its \textit{promotion chain} $p_{t_1}>p_{t_2}>\cdots>p_{t_{\ell}}$ is a saturated chain from a maximal element of $P$ to a minimal element of $P$ constructed as follows:
\begin{itemize}
\item First, $p_{t_1}=\lambda^{-1}(n)$ is the element with the largest $\lambda$-value;
\item Now suppose that $p_{t_i}$ has been assigned value $n$, let $p_{t_{i+1}}$ be the element covered by $p_{t_i}$ with the largest $\lambda$-value and swap the $\lambda$-values of $p_{t_i}$ and $p_{t_{i+1}}$ and of $p_{-t_i}$ and $p_{-t_{i+1}}$.
\end{itemize}
\end{defin}
Note that as the chain $p_{t_1}>p_{t_2}>\cdots>p_{t_{\ell}}$ records the movement of the value $n$ while $p_{-t_1}<p_{-t_2}<\cdots<p_{-t_{\ell}}$ records the movement of the value $-n$. 

The following lemma follows immediately from the above construction and is analogous to Lemma~\ref{lem:promotion} so we omit the proof here.

\begin{lemma}\label{lem:promotion-typeB}
The map $\varphi_{p_k}:\B(P\setminus\{p_k,p_{\bar k}\})\rightarrow\B(P)$ is injective, with image consisting of those $\lambda\in\B(P)$ whose promotion chain passes through $p_k$. 
\end{lemma}
\begin{defin}\label{def:insertion-typeB}
Let $P$ be a type $B$ poset on $2n$ elements. Then the \textit{insertion map} $f_{P}:B_n\rightarrow\B(P)$ is defined as $f_P(w)=\varphi_{p_{w(n)}}^{P^{(n)}}\circ \cdots \circ \varphi_{p_{w(2)}}^{P^{(2)}} \circ \varphi_{p_{w(1)}}^{P^{(1)}}$
where $P^{(k)}$ is the type $B$ poset on $2k$ elements obtained from $P$ by restricting to elements $p_{w(1)},p_{w(\bar 1)},\ldots,p_{w(k)},p_{w(\bar k)}.$
\end{defin}
\begin{ex}
We show an example of the insertion map $f_P$ (Definition~\ref{def:insertion-typeB}), where $w=4\bar213$ and $P$ is as the type $B$ poset in Figure~\ref{fig:insertion-typeB-example}. Notice that in step $\varphi_{p_1}$, the value $3$ is first assigned to $p_1$ which is a minimal element in $P$, and then moved all the way up to $p_{\bar1}$. 
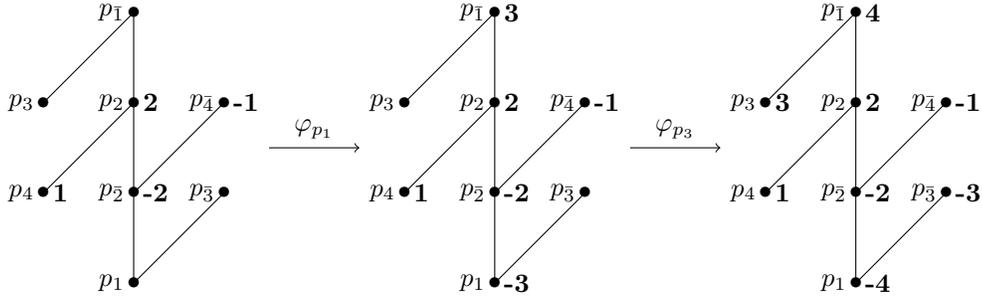
\begin{figure}[h!]
\centering
\begin{tikzpicture}[scale=1.20000000000000]
\node at (0,0) {$\bullet$};
\node[left] at (0,0) {$p_1$};
\node at (0,1) {$\bullet$};
\node[left] at (0,1) {$p_{\bar 2}$};
\node at (0,2) {$\bullet$};
\node[left] at (0,2) {$p_2$};
\node at (0,3) {$\bullet$};
\node[left] at (0,3) {$p_{\bar 1}$};
\node at (-1,2) {$\bullet$};
\node[left] at (-1,2) {$p_3$};
\node at (1,1) {$\bullet$};
\node[left] at (1,1) {$p_{\bar 3}$};
\node at (-1,1) {$\bullet$};
\node[left] at (-1,1) {$p_4$};
\node at (1,2) {$\bullet$};
\node[left] at (1,2) {$p_{\bar 4}$};
\draw(-1,2)--(0,3)--(0,0)--(1,1);
\draw(-1,1)--(0,2);
\draw(1,2)--(0,1);
\node[right] at (0,1) {\textbf{-2}};
\node[right] at (0,2) {\textbf{2}};
\node[right] at (-1,1) {\textbf{1}};
\node[right] at (1,2) {\textbf{-1}};
\node at (4,0) {$\bullet$};
\node[left] at (4,0) {$p_1$};
\node at (4,1) {$\bullet$};
\node[left] at (4,1) {$p_{\bar 2}$};
\node at (4,2) {$\bullet$};
\node[left] at (4,2) {$p_2$};
\node at (4,3) {$\bullet$};
\node[left] at (4,3) {$p_{\bar 1}$};
\node at (3,2) {$\bullet$};
\node[left] at (3,2) {$p_3$};
\node at (5,1) {$\bullet$};
\node[left] at (5,1) {$p_{\bar 3}$};
\node at (3,1) {$\bullet$};
\node[left] at (3,1) {$p_4$};
\node at (5,2) {$\bullet$};
\node[left] at (5,2) {$p_{\bar 4}$};
\draw(3,2)--(4,3)--(4,0)--(5,1);
\draw(3,1)--(4,2);
\draw(5,2)--(4,1);
\node[right] at (4,1) {\textbf{-2}};
\node[right] at (4,2) {\textbf{2}};
\node[right] at (3,1) {\textbf{1}};
\node[right] at (5,2) {\textbf{-1}};
\node[right] at (4,0) {\textbf{-3}};
\node[right] at (4,3) {\textbf{3}};
\node at (8,0) {$\bullet$};
\node[left] at (8,0) {$p_1$};
\node at (8,1) {$\bullet$};
\node[left] at (8,1) {$p_{\bar 2}$};
\node at (8,2) {$\bullet$};
\node[left] at (8,2) {$p_2$};
\node at (8,3) {$\bullet$};
\node[left] at (8,3) {$p_{\bar 1}$};
\node at (7,2) {$\bullet$};
\node[left] at (7,2) {$p_3$};
\node at (9,1) {$\bullet$};
\node[left] at (9,1) {$p_{\bar 3}$};
\node at (7,1) {$\bullet$};
\node[left] at (7,1) {$p_4$};
\node at (9,2) {$\bullet$};
\node[left] at (9,2) {$p_{\bar 4}$};
\draw(7,2)--(8,3)--(8,0)--(9,1);
\draw(7,1)--(8,2);
\draw(9,2)--(8,1);
\node[right] at (8,1) {\textbf{-2}};
\node[right] at (8,2) {\textbf{2}};
\node[right] at (7,1) {\textbf{1}};
\node[right] at (9,2) {\textbf{-1}};
\node[right] at (8,0) {\textbf{-4}};
\node[right] at (8,3) {\textbf{4}};
\node[right] at (9,1) {\textbf{-3}};
\node[right] at (7,2) {\textbf{3}};
\draw[->] (1.50000000000000,1.5)--(2.50000000000000,1.5);
\draw[->] (5.50000000000000,1.5)--(6.50000000000000,1.5);
\node[above] at (2.00000000000000,1.5) {$\varphi_{p_1}$};
\node[above] at (6.00000000000000,1.5) {$\varphi_{p_3}$};
\end{tikzpicture}
\caption{An example of the insertion map $f_P$ for type $B$ poset $P$}
\label{fig:insertion-typeB-example}
\end{figure}
\end{ex}

Note that if $\lambda$ is a type $B$ linear extension of $P$ and we insert $\pi_{\lambda}$ to $P$, then we obtain $\lambda$ back with no swaps applied.

\begin{defin}
Let $U\subset B_n$. Then the poset $P_U$ is defined on $p_1,\ldots,p_n,p_{\bar1},\ldots,p_{\bar n}$ such that $p_i<p_j$ if  $u^{-1}(i)<u^{-1}(j)$ for all $u\in U$.
\end{defin}

The proof of Proposition~\ref{prop:typeB-BW-lin-ext} is identical to that of Proposition~\ref{prop:generalized-BW-lin-ext}.

\begin{prop}\label{prop:typeB-BW-lin-ext}
For $U\subset B_n$, the convex set $\conv(U)$ is exactly $\{\pi_{\lambda}|\lambda\in\B(P_U)\}$. 
\end{prop}

As in Section~\ref{sec:permutations}, for $w\in B_n$, we also write $P_w=P_{\{\id,w\}}$ and write $wP_u=P_{w,wu}$ which is $P_u$ with its indices permuted according to $w$.

\begin{theorem}\label{thm:typeB}
For any $w,u\in B_n$, we have an injection:
$$\B(P_{\{\id,w,wu\}})\hookrightarrow \B(P_w)\times\B(wP_u),\text{ given by } \lambda\mapsto \big(f_{P_w}(\pi_\lambda),f_{wP_u}(\pi_\lambda)\big).$$
\end{theorem}
\begin{proof}
We use induction on $n$. Write $P=P_{\{\id,w,wu\}}$ for simplicity. The base case is $n=1$ where $B_1$ consists of two elements: $\id$, which is 1 in one-line notation, and $t_{1,\bar1}$, which is $\bar1$ in one-line notation. If $w=\id$, then $P=P_u$ so the above map is an isomorphism onto the second factor and if $w=\bar1$, then $P=P_w$ so the above map is an isomorphism onto the first factor. In both cases, we obtain an injection. Now assume that the claim is true for $n-1$. 

Let $A$ be the set of maximal elements of $P$. For $\lambda\in\B(P)$, $\lambda^{-1}(n)\in A$ so this provides us with a natural partition of $\B(P)$:
$$\B(P)=\bigsqcup_{p_a\in A}\{\lambda\in\B(P)\:|\:\lambda(p_a)=n\}\simeq\bigsqcup_{p_a\in A}\B(P\setminus\{p_a,p_{\bar a}\}).$$
Notice that unlike in type $A$, we now need to take out $p_a$ and $p_{\bar a}$ simultaneously, where $\lambda(p_{\bar a})=\bar n$ if $\lambda(p_a)=n$. It is possible that both $a$ and $\bar a$ belong to $A$, in which case $\B(P\setminus\{p_a,p_{\bar a}\})$ contributes twice in the above partition. However, these two parts will be treated differently.

For each $p_a\in A$, consider $w':\overline{[n]}\setminus\{w^{-1}(a),w^{-1}(\bar a)\}\rightarrow\overline{[n]}\setminus\{a,\bar a\}$ restricted from $w$ and $u'\overline{[n]}\setminus\{u^{-1}w^{-1}(a),u^{-1}w^{-1}(\bar a)\}\rightarrow\overline{[n]}\setminus\{w^{-1}(a),w^{-1}(\bar a)\}$ restricted from $u$. As the relative ordering of the other indices stays unchanged, we see that $P_{w}\setminus\{p_a,p_{\bar a}\}$ is $P_{w'}$, $wP_u\setminus\{p_a,p_{\bar a}\}$ is $u'P_{w'}$, and $P\setminus\{p_a,p_{\bar a}\}$ is $P_{\{\id,w',w'u'\}}$, viewing $w',u'\in B_{n-1}$. Thus, by the induction hypothesis, we have an injection
$$\bigsqcup_{p_a\in A}\B(P\setminus\{p_a,p_{\bar a}\})\hookrightarrow\bigsqcup_{p_a\in A}\B(P_w\setminus\{p_a,p_{\bar a}\})\times\B(wP_u\setminus\{p_a,p_{\bar a}\})$$ and it suffices to show that 
$$\varphi:\bigsqcup_{p_a\in A}\B(P_w\setminus\{p_a,p_{\bar a}\})\times\B(wP_u\setminus\{p_a,p_{\bar a}\})\rightarrow\B(P_w)\times\B(wP_u)$$
is an injection, where $\varphi$ acts as $\varphi_{p_a}$ on each part indexed by $p_a$. 

Assume to the contrary that $(\lambda,\tau)\in\B(P_w)\times\B(wP_u)$ lies in the image of both $\varphi_{p_a}$ and $\varphi_{p_{a'}}$. By Lemma~\ref{lem:promotion-typeB}, $\varphi_{p_a}$ is injective so $a\neq a'$. Moreover, both $p_a$ and $p_{a'}$ are in the promotion chains of both $\lambda$ and $\tau$ so in particular they are comparable in both $P_w$ and $wP_u$. Assume without loss of generality that $a<a'$. Then we must have $w^{-1}(a)<w^{-1}(a')$ and further $(wu)^{-1}(a)<(wu)^{-1}(a')$ so $p_a<_P p_{a'}$ in $P$, contradicting the maximality of $p_a$.
\end{proof}

The hull property then follows from Theorem~\ref{thm:typeB} and Proposition~\ref{prop:typeB-BW-lin-ext} via the same argument given in the proof of Corollary~\ref{cor:typeA}.

\begin{cor}
The Cayley graph $\cay(B_n)$ has the strong hull property.
\end{cor}

\begin{remark}
Our framework for signed permutations will not deduce that the type $D$ Cayley graph has the strong hull property. Since in type $D_n$ we no longer have inversions of the form $t_{i,\bar i}$, it is possible that $p_a$ and $p_{\bar a}$ are comparable in both $P_w$ and $wP_u$, but not comparable in $P_{\{\id,w,wu\}}$. 
\end{remark}

Writing $e_B(P)$ for the number of type $B$ linear extensions of the type $B$ poset $P$, we can deduce a direct analog of Sidorenko's inequality (\ref{eq:sidorenko}).

\begin{cor}
For any $w \in B_n$ we have
\[
e_B(P_w) e_B(P_{ww_0}) \geq 2^n n!.
\]
\end{cor}
\begin{proof}
By Proposition~\ref{prop:typeB-BW-lin-ext} we have that $e_B(P_w)=|\conv(\id,w)|$ and similarly that $e_B(P_{ww_0})=|\conv(\id,ww_0)|$. The hull property thus gives that the product is at least 
\[
|\conv(\id,w,ww_0)|=|B_n|=2^nn!.
\]
\end{proof}

\section{Right-angled Coxeter groups}\label{sec:right-angled}

A Coxeter group $W$ is called \emph{right-angled} if all entries $m_{ij}$ of the corresponding Coxeter matrix lie in $\{1,2,\infty\}$.

\begin{prop}
\label{prop:RA-meet-is-intersection}
Let $W$ be a right-angled Coxeter group and $u,v \in W$, then
\[
T_L(u \land_R v) = T_L(u) \cap T_L(v).
\]
\end{prop}
\begin{proof}
It suffices to prove the case when $u \land_R v = \id$ (so $T_L(u \land_R v)=\emptyset$).  In this case, suppose for the sake of contradiction that there is some $t \in T_L(u) \cap T_L(v)$; by Proposition~\ref{prop:inversions-from-prefixes} this means that there are reduced expressions $u=s_{i_1}\cdots s_{i_{k}}$ and $v=s_{j_1} \cdots s_{j_{\ell}}$ and indices $1 \leq k' \leq k, 1 \leq \ell' \leq \ell$ with
\[
as_{i_{k'}}a^{-1}=bs_{j_{\ell'}}b^{-1}
\]
where $a=s_{i_1}\cdots s_{i_{k'-1}}$ and $b=s_{j_1}\cdots s_{j_{\ell'-1}}$.  Since $W$ is right-angled, no two distinct simple reflections are conjugate (see, e.g., page 23 of \cite{Bjorner-Brenti}) so we must have
\[
s_{i_{k'}}=s_{j_{\ell'}}\coloneqq s.
\]

Now, let $C(s)$ denote the centralizer of $s$ in $W$; it follows from \cite{Brink} that $C(s)=W_J$, where $J$ consists of those simple reflections commuting with $s$.  At most one of $a,b$ lies in $C(s)$, otherwise we would have $s \leq_R u,v$, contradicting the fact that $u \land_R v=\id$, so suppose without loss of generality that $b \not \in C(s)=W_J$.  Since $b^{-1}a$ commutes with $s$, we have $a \in bW_J$, and thus $a \geq_R b^J \neq \id$.  But this means that $b^J \leq_R a \leq_R u$ and $b^J \leq_R b \leq_R v$, contradicting the fact that $u \land_R v = \id$.
\end{proof}

\begin{theorem}
\label{thm:right-angled-has-metric}
Let $W$ be a right-angled Coxeter group, then for any $u,v \in W$ we have an injection
\begin{align*}
    \phi: \conv(\id,u,v) &\to [\id,u]_R \times [\id,v]_R \\
    x &\mapsto (x \land_R u, x \land_R v).
\end{align*}
\end{theorem}
\begin{proof}
By Proposition~\ref{prop:convex-hull-from-inversions} we have
\begin{align*}
    \conv(\id,u,v)&=\{ x \in W \: | \: T_L(x) \subseteq T_L(u) \cup T_L(v)\} \\
    \conv(\id,u)&=\{ x \in W \: | \: T_L(x) \subseteq T_L(u)\} \\
    \conv(\id,v)&=\{ x \in W \: | \: T_L(x) \subseteq T_L(v)\} 
\end{align*}
and by Proposition~\ref{prop:RA-meet-is-intersection} we know that $\land_R$ is given by intersection of (left) inversion sets.  Thus the injectivity of $\phi$ follows from the obvious fact that a subset $X \subseteq U \cup V$ may be uniquely reconstructed given $X \cap U$ and $X \cap V$ for any sets $U,V$.
\end{proof}

\begin{cor}
\label{cor:right-angled-strong}
Let $W$ be right-angled, then $\cay(W)$ has the strong hull property.
\end{cor}

\section{Graphical arrangements}\label{sec:graphical}

An \emph{acyclic orientation} of a simple undirected graph $G$ is a choice of direction for each edge of $G$ such that there are no directed cycles.  We write $\ac(G)$ for the set of acyclic orientations of $G$, and endow $\ac(G)$ itself with a graph structure by taking $o,o' \in \ac(G)$ to be adjacent if the orientations differ on exactly one edge of $G$.

\begin{prop}
\label{prop:weak-order-acyclic-orientations}
Let $K_n$ denote the complete graph on $n$ vertices, then
\[
\ac(K_n) \cong \cay(\S_n).
\]
\end{prop}
\begin{proof}
Acyclic orientations $o$ of $K_n$ correspond to regions $R$ of the \emph{braid arrangement} $\mathcal{H}$, the hyperplane arrangement with hyperplanes $x_i=x_j$ for $1 \leq i<j \leq n$, with the orientation of the edge $\overline{ij}$ in $o$ indicating on which side of the hyperplane $x_i=x_j$ the region $R$ lies.  Furthermore, two regions $R,R'$ are adjacent if and only if the corresponding orientations $o,o'$ are adjacent in $\ac(K_n)$.  Now, it is well known (see, e.g. \cite{Humphreys}) that the graph of regions of a reflection arrangement is naturally identified with the Cayley graph of the reflection group, in this case $\S_n$, with its system of simple generators.
\end{proof}

In light of Corollary~\ref{cor:typeA}, which establishes the hull property for $\cay(\S_n) \cong \ac(K_n)$, it is natural to ask: for which graphs $G$ does $\ac(G)$ have the hull property? Equivalently, for which graphical hyperplane arrangements does the graph of regions have the hull property? Theorem~\ref{thm:only-complete-graph-has-metric} shows that there are essentially no new cases; this is meant to illustrate that the hull property is very special.

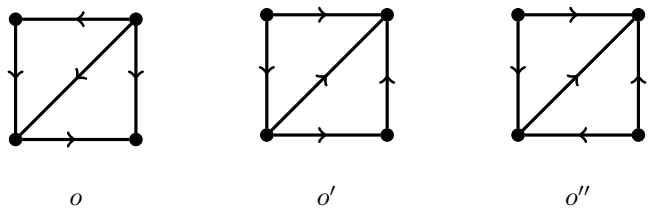
\begin{figure}[ht!]
\begin{center}
\begin{tikzpicture}[scale=.8]
\node[draw,shape=circle,fill=black,scale=0.5](a)[label=below: {}] at (0,0) {};
\node[draw,shape=circle,fill=black,scale=0.5](b)[label=below: {}] at (0,2) {};
\node[draw,shape=circle,fill=black,scale=0.5](c)[label=below: {}] at (2,0) {};
\node[draw,shape=circle,fill=black,scale=0.5](d)[label=below: {}] at (2,2) {};
\node at (1,-1) {$o$};

\begin{scope}[very thick,decoration={
    markings,
    mark=at position 0.5 with {\arrow{>}}}
    ] 
\draw[postaction={decorate}] (b)--(a);
\draw[postaction={decorate}] (a)->(c);
\draw[postaction={decorate}] (d)->(a);
\draw[postaction={decorate}] (d)->(c);
\draw[postaction={decorate}] (d)->(b);
\end{scope}
\end{tikzpicture}
\hspace{0.5in}
\begin{tikzpicture}[scale=.8]
\node[draw,shape=circle,fill=black,scale=0.5](a)[label=below: {}] at (0,0) {};
\node[draw,shape=circle,fill=black,scale=0.5](b)[label=below: {}] at (0,2) {};
\node[draw,shape=circle,fill=black,scale=0.5](c)[label=below: {}] at (2,0) {};
\node[draw,shape=circle,fill=black,scale=0.5](d)[label=below: {}] at (2,2) {};
\node at (1,-1) {$o'$};
\begin{scope}[very thick,decoration={
    markings,
    mark=at position 0.5 with {\arrow{>}}}
    ] 
\draw[postaction={decorate}] (b)--(a);
\draw[postaction={decorate}] (a)->(c);
\draw[postaction={decorate}] (a)->(d);
\draw[postaction={decorate}] (c)->(d);
\draw[postaction={decorate}] (b)->(d);
\end{scope}
\end{tikzpicture}
\hspace{0.5in}
\begin{tikzpicture}[scale=.8]
\node[draw,shape=circle,fill=black,scale=0.5](a)[label=below: {}] at (0,0) {};
\node[draw,shape=circle,fill=black,scale=0.5](b)[label=below: {}] at (0,2) {};
\node[draw,shape=circle,fill=black,scale=0.5](c)[label=below: {}] at (2,0) {};
\node[draw,shape=circle,fill=black,scale=0.5](d)[label=below: {}] at (2,2) {};
\node at (1,-1) {$o''$};

\begin{scope}[very thick,decoration={
    markings,
    mark=at position 0.5 with {\arrow{>}}}
    ] 
\draw[postaction={decorate}] (b)--(a);
\draw[postaction={decorate}] (c)->(a);
\draw[postaction={decorate}] (a)->(d);
\draw[postaction={decorate}] (c)->(d);
\draw[postaction={decorate}] (b)->(d);
\end{scope}
\end{tikzpicture}
\end{center}
\caption{Three orientations of $G_4$ which violate the triangle inequality in $\ac(G_4)$.  The sizes of the convex hulls of the pairs $(o,o'),(o',o''),$ and $(o,o'')$ are $4,2,$ and $9$, respectively; this violates the hull property.}
\label{fig:bad-graph-G_4}
\end{figure}

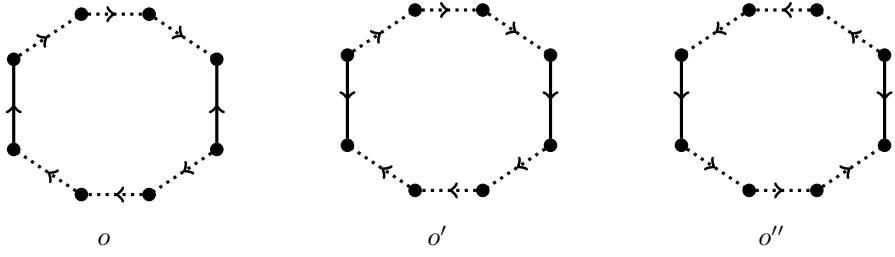
\begin{figure}[ht!]
\begin{center}
\begin{tikzpicture}[scale=.6]
\node[draw,shape=circle,fill=black,scale=0.5](a)[label=below: {}] at (0,0) {};
\node[draw,shape=circle,fill=black,scale=0.5](b)[label=below: {}] at (0,2) {};
\node[draw,shape=circle,fill=black,scale=0.5](c)[label=below: {}] at (1.5,3) {};
\node[draw,shape=circle,fill=black,scale=0.5](c2)[label=below: {}] at (3,3) {};
\node[draw,shape=circle,fill=black,scale=0.5](d)[label=below: {}] at (4.5,2) {};
\node[draw,shape=circle,fill=black,scale=0.5](e)[label=below: {}] at (4.5,0) {};
\node[draw,shape=circle,fill=black,scale=0.5](f2)[label=below: {}] at (1.5,-1) {};
\node[draw,shape=circle,fill=black,scale=0.5](f)[label=below: {}] at (3,-1) {};
\node at (2,-2) {$o$};

\begin{scope}[very thick,decoration={
    markings,
    mark=at position 0.5 with {\arrow{>}}}
    ] 
\draw[postaction={decorate}] (a)--(b);
\draw[postaction={decorate},dotted] (b)->(c);
\draw[postaction={decorate},dotted] (c)->(c2);
\draw[postaction={decorate},dotted] (c2)->(d);
\draw[postaction={decorate}] (e)->(d);
\draw[postaction={decorate},dotted] (e)->(f);
\draw[postaction={decorate},dotted] (f)->(f2);
\draw[postaction={decorate},dotted] (f2)->(a);
\end{scope}
\end{tikzpicture} \hspace{0.5in} \begin{tikzpicture}[scale=.6]
\node[draw,shape=circle,fill=black,scale=0.5](a)[label=below: {}] at (0,0) {};
\node[draw,shape=circle,fill=black,scale=0.5](b)[label=below: {}] at (0,2) {};
\node[draw,shape=circle,fill=black,scale=0.5](c)[label=below: {}] at (1.5,3) {};
\node[draw,shape=circle,fill=black,scale=0.5](c2)[label=below: {}] at (3,3) {};
\node[draw,shape=circle,fill=black,scale=0.5](d)[label=below: {}] at (4.5,2) {};
\node[draw,shape=circle,fill=black,scale=0.5](e)[label=below: {}] at (4.5,0) {};
\node[draw,shape=circle,fill=black,scale=0.5](f2)[label=below: {}] at (1.5,-1) {};
\node[draw,shape=circle,fill=black,scale=0.5](f)[label=below: {}] at (3,-1) {};
\node at (2,-2) {$o'$};

\begin{scope}[very thick,decoration={
    markings,
    mark=at position 0.5 with {\arrow{>}}}
    ] 
\draw[postaction={decorate}] (b)--(a);
\draw[postaction={decorate},dotted] (b)->(c);
\draw[postaction={decorate},dotted] (c)->(c2);
\draw[postaction={decorate},dotted] (c2)->(d);
\draw[postaction={decorate}] (d)->(e);
\draw[postaction={decorate},dotted] (e)->(f);
\draw[postaction={decorate},dotted] (f)->(f2);
\draw[postaction={decorate},dotted] (f2)->(a);
\end{scope}
\end{tikzpicture} \hspace{0.5in} \begin{tikzpicture}[scale=.6]
\node[draw,shape=circle,fill=black,scale=0.5](a)[label=below: {}] at (0,0) {};
\node[draw,shape=circle,fill=black,scale=0.5](b)[label=below: {}] at (0,2) {};
\node[draw,shape=circle,fill=black,scale=0.5](c)[label=below: {}] at (1.5,3) {};
\node[draw,shape=circle,fill=black,scale=0.5](c2)[label=below: {}] at (3,3) {};
\node[draw,shape=circle,fill=black,scale=0.5](d)[label=below: {}] at (4.5,2) {};
\node[draw,shape=circle,fill=black,scale=0.5](e)[label=below: {}] at (4.5,0) {};
\node[draw,shape=circle,fill=black,scale=0.5](f2)[label=below: {}] at (1.5,-1) {};
\node[draw,shape=circle,fill=black,scale=0.5](f)[label=below: {}] at (3,-1) {};
\node at (2,-2) {$o''$};

\begin{scope}[very thick,decoration={
    markings,
    mark=at position 0.5 with {\arrow{>}}}
    ] 
\draw[postaction={decorate}] (b)--(a);
\draw[postaction={decorate},dotted] (c)->(b);
\draw[postaction={decorate},dotted] (c2)->(c);
\draw[postaction={decorate},dotted] (d)->(c2);
\draw[postaction={decorate}] (d)->(e);
\draw[postaction={decorate},dotted] (f)->(e);
\draw[postaction={decorate},dotted] (f2)->(f);
\draw[postaction={decorate},dotted] (a)->(f2);
\end{scope}
\end{tikzpicture}

\end{center}
\caption{Three orientations of an $n$-cycle $C_n$ ($n \geq 4$) which violate the triangle inequality in $\ac(C_n)$; the sequences of dotted edges may be replace by any positive number of edges in the same direction.  The sizes of the convex hulls of the pairs $(o,o'),(o',o''),$ and $(o,o'')$ are $3,2^{n-2},$ and $2^n-2$, respectively; this violates the hull property.}
\label{fig:bad-cycles}
\end{figure}

\begin{theorem}
\label{thm:only-complete-graph-has-metric}
The graph $\ac(G)$ of acyclic orientations of $G$ has the hull property if and only if 
\[
\ac(G) \cong \ac(K_{n_1}) \times \cdots \times \ac(K_{n_k})
\]
for some $n_1, n_2, \ldots, n_k \in \{2,3,\ldots\}$.
\end{theorem}

\begin{proof}
If $\ac(G) \cong \ac(K_{n_1}) \times \cdots \times \ac(K_{n_k})$ it follows from Theorem~\ref{thm:typeA} and Proposition~\ref{prop:weak-order-acyclic-orientations} that $\ac(G)$ has the hull property, since the property is closed under taking direct products.

For the other direction, note first that if $\ac(G)$ has the hull property, then so does $\ac(H)$ for any induced subgraph $H$ of $G$.  Indeed in this case $\ac(H)$ is an induced subgraph of $\ac(G)$, as can be seen by fixing an acyclic orientation of $G \setminus H$ and directing all edges between $H$ and $G \setminus H$ away from $H$; then, varying the orientation only on $H$ gives an embedding $\ac(H) \hookrightarrow \ac(G)$ as an induced subgraph.  Now, if $\ac(H)$ did not have the hull property, with some orientations $o,o',o''$ violating the triangle inequality, pushing these forward to acyclic orientations of $G$ would violate the triangle inequality for $\ac(G)$.  Thus the hull property on $\ac(G)$ is closed under taking induced subgraphs of $G$.

Given orientations $o,o' \in \ac(G)$, it is not hard to see that 
\[
\conv(o,o')=\{o'' \in \ac(G) \: | \: \text{whenever $o,o'$ agree on an edge of $G$, so does $o''$} \}.
\]
Using this, it is easy to check the examples in Figures~\ref{fig:bad-graph-G_4} and \ref{fig:bad-cycles}.  These show that the graph $G_4$ obtained by gluing two triangles along an edge and the cycle graphs $C_n$ ($n \geq 4$) are \emph{bad}, where we say $H$ is bad if $\ac(H)$ does not possess the hull property.  By the previous paragraph, this means that any graph $G$ containing $G_4$ or $C_n$ as an induced subgraph is also bad.  If $H$ is not bad, we say it is $\emph{good}$.

If $H$ is disconnected, then $\ac(H) = \prod_i \ac(H_i)$ where the $H_i$ are the connected components, so $H$ is good if and only if all of the $H_i$ are good.  If $H$ has a cut vertex $v$, and the $H_i$ are the connected components of $H \setminus v$, then 
\[
\ac(H) \cong \prod_i \ac(H_i \cup \{v\}).
\]
Therefore it suffices to classify good 2-connected graphs $H$, and in particular to show that any 2-connected graph $H$ which does not contain an induced copy of $G_4$ or $C_n$ ($n \geq 4$) is a complete graph, and thus good by Theorem~\ref{thm:typeA}.

Let $H$ be a 2-connected graph $H$ which does not contain an induced copy of $G_4$ or $C_n$ ($n \geq 4$).  Any 2-connected graph on fewer than four vertices is a complete graph, so suppose $H$ has $N \geq 4$ vertices.  If $H$ is not a complete graph, then there exist two vertices $v,v'$ whose path distance from each other is exactly two.  Since $H$ is 2-connected, there is a cycle $C$ on at least four elements containing both $v,v'$, and such that $v,v'$ share a common neighbor $u \in C$.  As $H$ avoids induced copies of $C_n$ for $n \geq 4$, there must exist chords of $C$ in $H$ which triangulate $C$.  Since $\overline{vv'}$ is not a edge, one of these chords must yield a 4-cycle $v,u,v',w$ for some $w$ adjacent to $v,v'$, and furthermore the diagonal $\overline{uw}$ must be an edge of $H$.  But this creates an induced $G_4$ in $H$, a contradiction.
\end{proof}

\begin{remark}
It would be interesting to investigate the prevalence of the hull property within other families of graphs. The anonymous referee intriguingly notes that among the antipodal
isometric subgraphs of hypercubes that they were able to check, the only examples which satisfied the hull property were Cayley graphs of Coxeter groups. Another referee also suggests the case of (properly defined) acyclic orientations of signed graphs, which could generalize our type $B$ result.
\end{remark}

\section*{Acknowledgements}
We are grateful to Alex Postnikov, our advisor, for pointing out that the hull property could be interpreted in terms of the triangle inequality for a metric. We also thank Vic Reiner and Igor Pak for their comments.

\bibliographystyle{plain}
\bibliography{arxiv-v2}
\end{document}